\documentclass[10pt]{amsart}
\usepackage{amssymb,amsthm,amsmath}
\usepackage{ifthen}
\usepackage{graphicx}
\nonstopmode
 \numberwithin{equation}{section}
\newtheorem{thm}{Theorem}[section]

\newtheorem{lem}[thm]{Lemma}
\newtheorem{prop}[thm]{Proposition}
\newtheorem{defn}[thm]{Definition}

\newtheorem{prob}[thm]{Problem}
\theoremstyle{definition}

{\qed\bigskip}

\newcounter{alphabet}
\newcounter{tmp}

\newcommand{\vertiii}[1]{{\left\vert\kern-0.25ex\left\vert\kern-0.25ex\left\vert #1 
    \right\vert\kern-0.25ex\right\vert\kern-0.25ex\right\vert}}
  
\ifx\undefined\bysame
\newcommand{\bysame}{\leavevmode\hbox to3em{\hrulefill}\,}
\fi

\begin{document}
\baselineskip=21pt
\markboth{} {}

\bibliographystyle{amsplain}
\title[On a problem by Hans Feichtinger]
{On a problem by Hans Feichtinger}

\author{Radu Balan}
\author{Kasso A.~Okoudjou}
\author{Anirudha Poria} 

\address{Department of Mathematics, University of Maryland, College Park, MD 20742, USA}
\email{rvbalan@math.umd.edu}
\address{Department of Mathematics, University of Maryland, College Park, MD 20742, USA}
\email{kasso@math.umd.edu}
\address{Department of Mathematics, Indian Institute of Technology Guwahati, Assam 781039, India}
\email{a.poria@iitg.ernet.in}
\keywords{Modulation spaces, pseudodifferential operators,  time-frequency analysis, trace-class operators, Wilson bases.} 
\subjclass[2010]{Primary 45P05, 47B10; Secondary 42C15.}

\begin{abstract} 
In this paper, we solve a spectral problem about positive semi-definite trace-class pseudodifferential operators on modulation spaces which was posed by  H.~Feichtinger. Later, C.~Heil and D.~Larson rephrased the problem in the broader setting of positive semi-definite trace-class operators on a separable Hilbert space. Our solution consists in constructing a counterexample that solves Hans Feichtinger's problem by first solving this second problem.
\end{abstract}
\date{\today}
\maketitle

\def\BC{{\mathbb C}} \def\BQ{{\mathbb Q}}
\def\BR{{\mathbb R}} \def\BI{{\mathbb I}}
\def\BZ{{\mathbb Z}} \def\BD{{\mathbb D}}
\def\BP{{\mathbb P}} \def\BB{{\mathbb B}}
\def\BS{{\mathbb S}} \def\BH{{\mathbb H}}
\def\BE{{\mathbb E}}
\def\BN{{\mathbb N}}

\vspace{-.5cm}

\section{Introduction}

In this paper we answer the following  question posed by Feichtinger at an Oberwolfach mini-workshop on wavelets \cite{FeiOber}. 

\begin{prob}\label{fei-problem} Let $T$ be a positive semi-definite trace class operator on $L^2(\BR)$ given by 
$$Tf(x)=\int_{\BR}k(x,y)f(y)dy,$$ 
where $f\in L^2(\BR)$  and $k\in M^{1}(\BR^2)$, the so-called Feichtinger algebra. Suppose that 
$$T=\sum_{k=1}^\infty h_k\otimes \overline{h_k},$$ 
where  $\{h_k\}_{k=1}^{\infty}  \subset L^2(\BR)$ is a set of orthogonal eigenfunctions of $T$ corresponding to the eigenvalues $\{\|h_k\|_2^2\}_{k=1}^{\infty}$, such that $\|h_k\|_{M^1(\BR)}<\infty$, and the bar denotes the complex conjugation. 
In particular, $\text{Trace}(T)=\sum_{k=1}^{\infty}\|h_k\|_2^2<\infty.$ 

Must we have: $\sum_{k=1}^{\infty}\|h_k\|_{M^{1}(\BR)}^2<\infty?$
\end{prob}

Heil and Larson later put the problem in the broader setting  of positive semi-definite trace-class operators on a separable Hilbert space $\BH$  \cite{hei08}. To state this generalization we first set some notations. Let $\BH$ be a separable Hilbert space and choose  an orthonormal basis $\{w_n\}_{n  \geq 1}$ for $\BH$. We  define a subspace $\BH^1$ of $\BH$ by 
\begin{equation}\label{eq01}
\BH^1=\Big\{ f \in \BH : \vertiii{f}:=\sum_{n=1}^{\infty} |\langle f, w_n\rangle| < \infty \Big\}.
\end{equation}
It follows that $\vertiii{w_n}=\Vert w_n \Vert=1$ for every $n$, and that if  $f \in \BH^1$ then $f=\sum_{n=1}^{\infty}\langle f,w_n \rangle w_n$, with convergence of this series in $both$ norms $\Vert \cdot \Vert$ and $\vertiii{\cdot}$.  

We define an operator  $T:\BH \rightarrow \BH$ by 
\begin{equation}\label{eq02}
T=\sum_{m=1}^\infty \sum_{n=1}^\infty c_{mn} (w_m \otimes \overline{w_n}),
\end{equation}
where the  scalars $c_{mn}$ are such that \[ \sum_{m=1}^\infty \sum_{n=1}^\infty |c_{mn}|< \infty\]
and the tensor product $w_m\otimes \overline{w_n}$ maps linearly $\BH$ to $\BH$ via
\[ f\in \BH ~\mapsto ~ w_m\otimes \overline{w_n} (f) = \langle f, w_n\rangle w_m. \]
It is easy to see that $T \in  \mathcal{I}_1$, the space of all trace-class operators, with  

\[\Vert T \Vert_{\mathcal{I}_1} \leq \sum_{m=1}^\infty \sum_{n=1}^\infty \Vert c_{mn} (w_m \otimes 
\overline{w_n})\Vert_{\mathcal{I}_1}=\sum_{m=1}^\infty \sum_{n=1}^\infty | c_{mn}|< \infty. \]
In addition, note that  the series defining $T$ converges not only in the strong operator topology and operator norm, but also in trace-class norm.

 Now suppose that the operator  $T$ given by~\eqref{eq02} is positive semi-definite.  Let $\{h_n\}_{n \geq 1}$ be an orthonormal basis of eigenvectors of $T$ and $\{ \lambda_n \}_{ n \geq 1} \subset [0, \infty)$ be the corresponding eigenvalues. It follows that 
\begin{equation}\label{eq03}
T=\sum_{n=1}^{\infty} \lambda_n (h_n \otimes \overline{h_n})=\sum_{n=1}^\infty g_n \otimes 
\overline{g_n},
\end{equation}
where $g_n=\lambda_n^{1/2} h_n$. In addition, \[\Vert T \Vert_{\mathcal{I}_1}=\sum_{n=1}^{\infty} \lambda_n=\sum_{n=1}^{\infty} \lambda_n \Vert h_n \Vert^2 < \infty.\]
Heil and Larson's generalization of Problem~\ref{fei-problem} is the following question \cite{hei08}.

\begin{prob}\label{problem1}
With the above notations, must we have 
\begin{equation}\label{eq04}
\sum_{n=1}^{\infty} \lambda_n \vertiii{h_n}^2 < \infty ?
\end{equation}
\end{prob}

In Section~\ref{sec3} we show that the solution to each of these problems is negative by providing counterexamples for each of them. But first, we provide some necessary background in Section~\ref{sec2}

\section{Preliminaries}\label{sec2} In this section we recall the definition of the modulation spaces and some of their properties. In the second half of the section, we introduce two classes of trace-class operators that capture the behaviors of the operators in Problems~\ref{fei-problem} and~\ref{problem1}.

\subsection{Modulation spaces}\label{subsec1.1}
Let  $g\in \mathcal{S}(\BR)$ be a function in the Schwartz space of smooth and rapidly decaying functions, e.g., $g(x)=e^{-\pi x^2}$, and let $1\leq p \leq \infty$. We say that a tempered distribution $f$ is in the modulation space $M^{p}(\BR)$ if and only if 
$$\|f\|_{M^{p}}^p:=\iint_{\BR^2}|V_gf(x, \omega)|^pdxd\omega< \infty,$$ 
with the usual modification for $p=\infty$, where 
$$ V_g f (x,\omega)=\int_{\BR} f(t) \overline{g(t-x)} e^{-2 \pi i \omega t} dt$$ is the $short$-$time$ $Fourier$ $transform$ (STFT) of a function $f $ with respect to $g$.  A simple application of the Plancherel formula  shows that if $f\in L^2(\BR)$ then $$\|V_gf\|_{L^{2}(\BR^2)}^2=\iint_{\BR^2} |V_gf(x, \omega)|^2dxd\omega =\|g\|_2^2\|f\|_2^2.$$ Consequently, $V_g$ is a multiple of an isometry from $L^2(\BR)$ into $L^2(\BR^2)$ and $M^2(\BR)=L^2(\BR), $ \cite{gro01}.  The other modulation space that will be of interest in the sequel is $M^1(\BR)$, which is also known as the Feichtinger algebra \cite{fei03, gro01}. In particular, we note that \[ \mathcal{S}(\BR) \subset M^1(\BR) \subset M^2(\BR)=L^2(\BR) \subset M^\infty(\BR) \subset \mathcal{S'}(\BR). \]

We also need a discrete characterization of $L^2$ and $M^1$. Such a characterization exists for all the modulation spaces in terms of the so-called Wilson basis, see \cite{dau91, fei92, wil87}. In particular, it is known that there exists an orthonormal basis $\mathcal{W}:=\{w_n\}_{n \geq 1}$ for $L^2(\BR)$ where for each $n\geq 1$, $w_n \in M^1(\BR)$. In addition, for $1\leq p \leq \infty$ and for all $f \in M^p$, 
\[ f=\sum_{n\geq1}  \langle f, w_{n} \rangle w_{n},  \]
where the series converges unconditionally in the norm of $M^p$ if $1\leq p<\infty$, and is weak$^*$ convergent if $p =\infty$. Moreover, 
\begin{equation*}
\Vert f \Vert_{M^p}=\bigg( \sum_{n\geq 1} |\langle f, w_{n} \rangle|^p \bigg)^{1/p}
\end{equation*}
is an equivalent norm for $M^p$; we refer to \cite[Theorem 8.5.1]{gro01} for details. In the sequel,  we shall  only be interested in $p=1,$ and $p=2$. In the latter case, $\{w_n\}_{n\geq 1}$ is an orthonormal basis for $L^2(\BR).$

It is trivial to extend these characterizations to modulation spaces defined on $\BR^d$. In particular, one defines a Wilson orthonormal basis for  $L^2(\BR^2)$ by taking the  tensor product of $1$-dimensional Wilson ONBs. For example, $\{W_{n, m}: n, m \geq 1\} \subset L^{2}(\BR^2)$ is given by  
$$W_{n, m}(x, y):=w_n\otimes \overline{w_m} (x, y) =w_{n}(x)\overline{w_m(y)}, \quad n,m \geq 1,$$
and it acts by
$$W_{n, m}(f)=\langle f,w_m \rangle w_n= \bigg( \int_{\BR} f(y) \overline{w_m(y)} dy \bigg) w_n.$$ 
 In addition,  $\{W_{n, m}: n, m \geq 1\}$ is an unconditional basis for $M^1(\BR^2)$. 

 Let $T: L^2(\BR)\to L^2(\BR)$ be a compact integral operator associated with the kernel $k\in M^1(\BR^2)\subset L^2(\BR^2)\cap L^1(\BR^2)$ and defined by $$Tf(x)=\int_{\BR}k(x,y)\, f(y)dy.$$ Then, $T$ is a trace-class operator \cite{hei08}, and

 \begin{equation}\label{eq3}
k=\sum_{m,n \geq 1} \langle k,  W_{m,n} \rangle  W_{m,n},
\end{equation}
with convergence of the series in the $M^1$-norm. In addition, 
\begin{equation}\label{eq4}
\Vert k \Vert_{M^1}=\sum_{m,n \geq 1} |\langle k,  W_{mn} \rangle| < \infty.
\end{equation}
 It now follows that for $f\in L^2(\BR)$,
\begin{eqnarray*}
Tf =  \sum_{m,n \geq 1} \langle k,  W_{mn} \rangle (w_m \otimes \overline{w_n}) (f)
 =\sum_{m,n \geq 1} \langle k,  W_{mn} \rangle (W_{m,n}) (f).
\end{eqnarray*}
The discrete version of the integral operator $T$ is given by the matrix $K=(\langle k,  W_{m,n} \rangle)_{m, n \geq 1}$, or equivalently 
\begin{equation}\label{eq6}
T=\sum_{m,n \geq 1} \langle k,  W_{m, n} \rangle W_{m,n}. 
\end{equation}
Suppose in addition that $T$ is positive semi-definite. Then, by the spectral theorem, 
$$T=\sum_{k=1}^{\infty}\lambda_k t_k\otimes \overline{t_k}=
\sum_{k=1}^\infty h_k\otimes \overline{h_k},$$ 
where $\{\lambda_k\}_{k=1}^{\infty}\subset (0, \infty)$ is the set of eigenvalues of $T$ and $\{t_k\}_{k=1}^\infty$ is an orthonormal basis of corresponding eigenfunctions, and  $h_k=\sqrt{\lambda_k}t_k$ for each $k\geq 1$. It was proved in \cite{CorFeiLue08, hei08} that $h_k \in M^1(\BR)$.

 \subsection{Type $A$ and type $B$ operators}\label{subsec1.2}
 
 Let  $\BH$  denote an infinite-dimensional separable Hilbert space, with norm $\Vert \cdot \Vert$ and inner product $\langle \cdot , \cdot \rangle$.
Let $\mathcal{I}_1 \subset \mathcal{B}(\BH)$  be the subspace of trace-class operators. A positive semi-definite operator $T$ belongs to $ \mathcal{I}_1$ if and only if 
$$\|T\|_{\mathcal{I}_1}=\sum_{n=1}^\infty \lambda_n(T)<\infty, $$ 
where $\{\lambda_n(T)\}_{n\geq 1}$ is the set of eigenvalues of $T$ arranged  in a decreasing order and repeated according to multiplicity. For a detailed study on trace-class operators see \cite{dun88, sim79}.

We fix now an orthonormal basis $\{w_n\}_{n \geq 1}$ for $\BH$, once and for all. This basis induces the norm $\vertiii{\cdot}$ on the dense subset $\BH^1$ introduced in (\ref{eq01}), and repeated here for the convenience of the reader: 
\[ \vertiii{f}=\sum_{n=1}^{\infty} |\langle f, w_n\rangle|, \quad \BH^1=\Big\{ f \in \BH : \sum_{n=1}^{\infty} |\langle f, w_n\rangle| < \infty \Big\}.\]

\begin{defn}\label{def-types}
An operator $T$ given by~\eqref{eq02} is of \emph{Type $A$} with respect to the orthonormal basis $\{w_n\}_{n \geq 1}$ if, for an orthogonal set of eigenvectors $\{g_n\}_{n \geq 1}$ of $T$ such that $T=\sum_{n =1}^\infty g_n \otimes \overline{g_n}$, with convergence in the strong operator topology, we have that 
\begin{equation*}
\sum_{n=1}^{\infty}  \vertiii{g_n}^2 < \infty.
\end{equation*} 
\end{defn}

\begin{defn}\label{def-types}
An operator $T$ given by~\eqref{eq02} is of \emph{Type $B$} with respect to the orthonormal basis $\{w_n\}_{n \geq 1}$  if there is some sequence of vectors $\{v_n \}_{n \geq 1}$ in $\BH$ such that $T=\sum_{n=1}^\infty v_n \otimes 
\overline{v_n}$ with convergence in the strong operator topology and we have that
\[ \sum_{n=1}^{\infty}  \vertiii{v_n}^2 < \infty. \] 
\end{defn}

It is clear that if $T$ is of Type $A$ then it is of Type $B$. However, it was shown in \cite[Example 2.2]{hei08}  that  not every positive trace-class operator is of Type $A$ or Type $B$, even when the operator is finite-rank.

Problem~\ref{problem1} can now be reformulated as follows. 
\begin{prob}\label{problem2}
If $T$ is of Type $B$ with respect to an orthonormal basis $\{w_n\}_{n \geq 1}$ , must
it be of Type $A$ with respect to the same ONB $\{w_n\}_{n \geq 1}$ $?$
\end{prob}

 \section{Main results}\label{sec3}

We answer negatively Problems~\ref{problem1} and~\ref{problem2} by constructing a counterexample for the complex Hilbert space $\BH$, in Proposition~\ref{proposition1}. This example is then modified to generate an example when the  Hilbert space $\BH$ is over the real field, in Proposition~\ref{proposition2}. From there, we answer the Feichtinger original problem in Theorem~\ref{mainresult}.

\begin{prop}\label{proposition1} Let  $\BH=\ell^2(\{1,2,...\})$, and choose $p > 1$. Let $\{w_\ell\}_{\ell=1}^{\infty}$ denote  the standard orthonormal basis of $\BH$, i.e., $w_{\ell}=\delta_{\ell}$. Then $\BH^1=\ell^1(\{1,2,...\})$. For each $n\geq 1$, let $\{e_{n,k}\}_{k=0}^{n-1}$ be the Fourier ONB of $\BC^n$ defined by \[ e_{n,k}=\frac{1}{\sqrt{n}} \left( e^{- \frac{2 \pi i k \ell}{n}} \right)_{\ell=0}^{n-1} =\frac{1}{\sqrt{n}} \left( 1, e^{- \frac{2 \pi i k}{n}}, e^{- \frac{4 \pi i k}{n}},..., e^{- \frac{2 \pi i k (n-1)}{n}} \right)^T, \]  
and consider the $n\times n$ matrix $T_n$ given by 
 \[ T_n=\sum_{k=0}^{n-1} \lambda_{n,k}(e_{n,k} \otimes \overline{e_{n,k}} )= \frac{1}{n^3} \sum_{k=0}^{n-1} \left(1+\frac{k}{n^p}\right) (e_{n,k} \otimes \overline{e_{n,k}} ) \in \BC^{n \times n}, \]   
where  $\lambda_{n,k}=\frac{1}{n^3} \left( 1+\frac{k}{n^p} \right)$.
We define an  infinite block-diagonal matrix $T$ by   $$T=T_1 \oplus T_2 \oplus ... \oplus T_n \oplus...$$ 
Then, $T$ is a positive semi-definite trace-class operator of Type $B$ but not of Type $A$ with respect to the orthonormal basis $\{w_\ell\}$.
\end{prop}

\begin{proof}

By construction, the blocks $T_n$ that make up $T$ are pairwise orthogonal. Furthermore, for each $n\geq 1$, the spectrum of $T_n$ consists of simple eigenvalues $\lambda_{n, k}$ with corresponding eigenvectors $e_{n, k}$ for $k=0, \hdots, n-1$. Consequently, for each $n\geq 1$, and each $k \in \{0, \hdots, n-1\}$, $e_{n,k}$ generates  a one-dimensional eigenspace of  $T$ corresponding to the eigenvalue $\lambda_{n,k}$. It is clear that $T$ is positive semi-definite. Since $\Vert e_{n,k} \Vert_2=1$ and $T=\bigoplus_{n=1}^{\infty} \sum_{k=0}^{n-1} \lambda_{n,k}(e_{n,k} \otimes \overline{e_{n,k}})$, we see that 

\begin{align*}
\|T\|_{\text{op}}&\leq \sum_{n = 1}^\infty \sum_{k=0}^{n-1}\frac{1}{n^3}  \left(1+\frac{k}{n^p}\right) \|e_{n,k} \otimes \overline{e_{n,k}}\|_{\text{op}}\\
&=\sum_{n = 1}^\infty \sum_{k=0}^{n-1}\frac{1}{n^3}  \left(1+\frac{k}{n^p}\right) \|e_{n,k} \|\\
&= \sum_{n = 1}^\infty \sum_{k=0}^{n-1}\frac{1}{n^3}  \left(1+\frac{k}{n^p}\right) <\infty.
\end{align*}
Furthermore, since $p>1$, we see that 
\begin{eqnarray*}
\Vert T  \Vert_{\mathcal{I}_1}= \mathrm{trace}(T) &=& \sum_{n = 1}^\infty \sum_{k=0}^{n-1}\frac{1}{n^3}  \left(1+\frac{k}{n^p}\right) \\
&=& \sum_{n = 1}^\infty \frac{1}{n^3} \left( n+ \frac{n(n-1)}{2n^p} \right)\\
&<& \infty.
\end{eqnarray*}
Hence $T$  is a well-defined  trace-class operator on $\BH$. 

We now show that $T$ is of Type $B$. To this end we observe that for each $n\geq 1$, $\sum_{k=0}^{n-1} e_{n,k} \otimes \overline{e_{n,k}}=I_n$, where $I_n$ denotes the identity of order $n$. Then  
\begin{eqnarray*}
 T_n &=& \frac{1}{n^3} \sum_{k=0}^{n-1} \left(1+\frac{k}{n^p}\right) (e_{n,k} \otimes \overline{e_{n,k}} ) \\
&=&  \frac{1}{n^3} \sum_{k=0}^{n-1} (e_{n,k} \otimes \overline{e_{n,k}}) +\frac{1}{n^{3+p}} \sum_{k=0}^{n-1} k (e_{n,k} \otimes \overline{e_{n,k}}) \\
&=&  \frac{1}{n^3} I_n +\frac{1}{n^{3+p}} \sum_{k=0}^{n-1} k (e_{n,k} \otimes \overline{e_{n,k}}).
\end{eqnarray*}
Thus  $T$ can be written as 
\begin{eqnarray*}
T &=& \bigoplus_{n \geq 1} T_n=\bigoplus_{n \geq 1} \left( \frac{1}{n^3} I_n +\frac{1}{n^{3+p}} \sum_{k=0}^{n-1} k (e_{n,k} \otimes \overline{e_{n,k}}) \right) \\
&=& \bigoplus_{n \geq 1} \left( \frac{1}{n^3} I_n  \right) + \bigoplus_{n \geq 1}  \frac{1}{n^{3+p}} \sum_{k=0}^{n-1} k (e_{n,k} \otimes \overline{e_{n,k}}) \\
&=& \bigoplus_{n \geq 1} \frac{1}{n^3} \sum_{k=1}^n (w_{\frac{n(n-1)}{2}+k} \otimes \overline{w_{\frac{n(n-1)}{2}+k}}) + \bigoplus_{n \geq 1}  \frac{1}{n^{3+p}} \sum_{k=0}^{n-1} k (e_{n,k} \otimes \overline{e_{n,k}}).
\end{eqnarray*}
Then we have
\[\vertiii{w_{\frac{n(n-1)}{2}+k}}=1, \quad \vertiii{e_{n,k}}=\sqrt{n},\]
and
\begin{eqnarray*}
&& \sum_{n \geq 1} \frac{1}{n^3} \cdot \sum_{k=1}^n 1^2+ \sum_{n \geq 1} \frac{1}{n^{3+p}} \sum_{k=0}^{n-1} k \cdot (\sqrt{n})^2 \\
&=& \sum_{n \geq 1} \left( \frac{1}{n^2} + \frac{n-1}{2n^{1+p}} \right) < \infty, \quad \mathrm{for \; any \;} p>1.
\end{eqnarray*}
Hence, $T$ is of Type $B$ with respect to $\{w_\ell\}_{\ell \geq 1}$. 

We now show that $T$ is not of Type $A$ with respect to $\{w_\ell\}_{\ell}$. The key point is that $T$ has only one-dimensional eigenspaces, so 
$$\sum_{n=1}^{\infty} \sum_{k=0}^{n-1} \lambda_{n,k}(e_{n,k} \otimes \overline{e_{n,k}} )= \sum_{n=1}^{\infty} \frac{1}{n^3} \sum_{k=0}^{n-1} \left(1+\frac{k}{n^p}\right) (e_{n,k} \otimes \overline{e_{n,k}} )$$ 
is the unique decomposition of $T$ as a sum of rank one projections generated by orthogonal eigenfunctions of $T$.  Note again that $\vertiii{ e_{n,k} }=\sqrt{n},$ and 
\[ \lambda_{n,k} \vertiii{e_{n,k}}=\frac{1}{n^3} \left( 1+\frac{k}{n^p} \right) \cdot \sqrt{n} < \infty. \]
However, 
\begin{eqnarray*}
\sum_{n = 1}^\infty \sum_{k=0}^{n-1} \lambda_{n,k} \vertiii{e_{n,k}}^2 &=& \sum_{n = 1}^\infty \frac{1}{n^2} \sum_{k=0}^{n-1} \left(1+\frac{k}{ n^p} \right) \\
&=& \sum_{n = 1}^\infty  \frac{1}{n^2}  \left( n+ \frac{n(n-1)}{2n^p} \right)\\
&\geq & \sum_{n = 1}^\infty \frac{1}{n}=\infty.
\end{eqnarray*} 
\end{proof}

We can modify the counterexample in Proposition~\ref{proposition1} to deal with the case of a real Hilbert space $\BH$. This amounts to using a real-valued ONB for $\BR^n$ instead of the Fourier ONB $\{e_{n,k}\}_{k=0}^{n-1}$. 
For this let $\{h_{n, k}\}_{k=0}^{n-1}$ denote the  Hartley ONB basis for $\BR^n$ (see \cite{wic03}), where
$$h_{n,k}=\frac{1}{\sqrt{n}} \left( \cos \left(\frac{2 \pi k l}{n}\right)+ \sin\left(\frac{2 \pi k l}{n}\right) \right)_{l=0}^{n-1}=\sqrt{\frac{2}{n}}\left( \cos\left( \frac{2 \pi k l}{n}- \frac{\pi}{4}\right)\right)_{l=0}^{n-1}.$$
Thus
\[ \sum_{k=0}^{n-1} h_{n,k}\otimes \overline{h_{n,k}} = \sum_{k=0}^{n-1} h_{n,k}\otimes h_{n,k} = I_n, \]
where $I_n$ denotes the identity of order $n$ in $\BR^n$.

\begin{lem}\label{lem1}
For a fixed $n \geq 1$ and each $0 \leq k \leq n-1$ we have
\begin{equation}\label{eq11b}
\sqrt{\frac{n}{2}} \leq \vertiii{h_{n,k}}=\frac{1}{\sqrt{n}}\sum_{l=0}^{n-1} \left| \cos \left(\frac{2 \pi k l}{n}\right)+ \sin\left(\frac{2 \pi k l}{n}\right)  \right| \leq \sqrt{n}.
\end{equation}
\end{lem}

\begin{proof}
Denote by $S_n$ the set
\begin{equation*}\label{eq8}
S_n:=\bigg\{\frac{2 \pi k}{n}: 0\leq k \leq n-1 \bigg\}.
\end{equation*}
It is easy to see that for each $0 \leq l \leq n-1$ we have
\begin{equation*}
S_n=\bigg\{ \frac{2 \pi k l}{n} \;( \mathrm{mod} \; 2\pi): 0\leq k \leq n-1  \bigg\}=\bigg\{ - \frac{2 \pi k}{n} \;( \mathrm{mod} \; 2\pi): 0\leq k \leq n-1  \bigg\}.
\end{equation*} 
Let $E:=\sum_{x \in S_n}\left|\cos x+\sin x \right|.$
Then
\begin{eqnarray}\label{eq12}
\nonumber 2E &=&\sum_{x \in S_n}\left|\cos x+\sin x \right|+\sum_{-x \in S_n}\left|\cos x+\sin x \right|\\ \nonumber
&=& \sqrt{2} \sum_{k=0}^{n-1} \left| \cos\left( \frac{2 \pi k}{n} - \frac{\pi}{4} \right) \right| + \sqrt{2} \sum_{k=0}^{n-1} \left| \cos\left( \frac{2 \pi k}{n} + \frac{\pi}{4} \right) \right|\\ 
&=& \sqrt{2} \sum_{k=0}^{n-1} \left[ \left| \cos\left( \frac{2 \pi k}{n} - \frac{\pi}{4} \right) \right| +  \left| \sin\left( \frac{2 \pi k}{n} - \frac{\pi}{4} \right) \right| \right].
\end{eqnarray}
Now for each $x \in \BR$,
\begin{eqnarray*}
&& (|\sin x|+ |\cos x|)^2 =|\sin x|^2+ |\cos x|^2+2 |\sin x \cos x|=1+|\sin 2x| \geq 1,\\
&& \Rightarrow \sqrt{2} \geq |\sin x|+ |\cos x| \geq 1.
\end{eqnarray*}
It follows from~\eqref{eq12} that  $n\geq E\geq \frac{n}{\sqrt{2}}$ and therefore (\ref{eq11b}). 
\end{proof}

\begin{prop}\label{proposition2} Let $\BH=\ell^2(\{1,2,...\})$, and choose $p > 1$.  Let $\{w_\ell\}_{\ell=1}^{\infty}$ denote  the standard orthonormal basis of $\BH$, i.e., $w_{\ell}=\delta_{\ell}$.
For each $n\geq 1$  let $T_n$ denote the $n\times n$ matrix given by
\[ T_n= \frac{1}{n^3}\sum_{k=0}^{n-1} \left(1+\frac{k}{n^p}\right) (h_{n,k} \otimes h_{n,k}) \in \BR^{n \times n}. \] 
We define an  infinite block-diagonal matrix $T$ by   
$$T=T_1 \oplus T_2 \oplus ... \oplus T_n \oplus...$$ 
Then, $T$ is a positive semi-definite trace-class operator of Type $B$ but not of Type $A$ with respect to the orthonormal basis $\{w_\ell\}_{\ell \geq 1}$. 

\end{prop}

\begin{proof}
The proof is almost identical to that of Proposition~\ref{proposition1} where the Fourier ONB vectors $e_{n,k}$
are replaced by the Hartley ONB vectors $h_{n,k}$ and the estimate $\vertiii{e_{n,k}}=\sqrt{n}$ 
is replaced by $\sqrt{\frac{n}{2}}\leq \vertiii{h_{n,k}} \leq \sqrt{n}$,  cf. Lemma \ref{lem1}.
\end{proof}

We can now give an answer to  Feichtinger's question, i.e., Problem~\ref{problem1}.

\begin{thm}\label{mainresult} Suppose that $\{w_n\}_{n \geq 1}$ is a Wilson orthonormal basis for $L^2(\BR)$ with $g \in M^1(\BR)$. Let $p>1$, and for each $n\geq 1$ set $\lambda_{n,k}=\frac{1}{n^3}(1+\frac{k}{n^p}).$

For fixed $n \geq 1$ and each $0 \leq k \leq n-1$, let  $h_{n,k} \in L^{2}(\BR)$ where 
\[ h_{n,k}=\frac{1}{\sqrt{n}} \sum_{l=0}^{n-1} \left( \cos \left( \frac{2 \pi k l}{n} \right)+ \sin\left( \frac{2 \pi k l}{n} \right) \right) w_{\frac{n(n-1)}{2}+l+1}. \] 
Let $T$ be the operator defined by  $$T=\sum_{n=1}^{\infty}\sum_{k=0}^{n-1}\lambda_{n,k}h_{n,k}\otimes h_{n,k}.$$ The following statements hold:

\begin{enumerate}
\item[(i)] $\{h_{n,k}: 0 \leq k \leq {n-1}, n\geq 1 \}$ is an orthonormal basis for $L^2(\BR)$. 
\item[(ii)]  $T$ is a positive semi-definite trace-class operator on $L^2(\BR)$ that provides a counter-example to Problem~\ref{problem1}. 
\end{enumerate}
\end{thm}

\begin{proof}
(i) It is easy to see that  for each $n\geq 1,$ $\{h_{n,k}\}_{k=0}^{n-1}$ is  an orthogonal set in $L^2(\BR)$. Indeed, $\langle h_{n,k}, h_{n',k'} \rangle =0$, for $n \neq n'$. Furthermore, since $\langle w_n, w_m \rangle=\delta_{n,m}$ we have that $\|h_{n,k}\|=1$ for all $n\geq 1,$ and $k\in \{0, 1, \hdots, n-1\}$. \\
(ii)  It is also easy to see that $T$ is a well-defined operator on $L^2(\BR)$. In fact, the series defining $T$ converges in the operator norm. Furthermore, since $\|h_{n,k}\otimes h_{n,k}\|_{\mathcal{I}_{1}}=1$, it follows that 
\begin{align*}
\|T\|_{\mathcal{I}_{1}}&=\sum_{n=1}^\infty \sum_{k=0}^{n-1} \lambda_{n,k} \\
&= \sum_{n=1}^\infty \frac{1}{n^3} \sum_{k=0}^{n-1}(1+\frac{k}{n^p}) \\
&= \sum_{n=1}^\infty \frac{1}{n^3} \left( n+ \frac{n(n-1)}{2n^p}\right)< \infty. 
\end{align*} 
Consequently, $T$ is a  trace-class operator.

By Lemma \ref{lem1},
\begin{eqnarray*}
\Vert h_{n,k} \Vert_{M^1} &=& \sum_{m=1}^\infty |\langle h_{n,k}, w_m \rangle| \\
&=& \frac{1}{\sqrt{n}}\sum_{m=1}^\infty \left| \Bigg\langle \sum_{l=0}^{n-1} \left( \cos\left( \frac{2 \pi k l}{n} \right)+ \sin\left( \frac{2 \pi k l}{n} \right) \right) w_{\frac{n(n-1)}{2}+l}, w_m \Bigg\rangle \right| \\
&=& \frac{1}{\sqrt{n}} \sum_{l=0}^{n-1} \left| \cos\left( \frac{2 \pi k l}{n} \right)+ \sin\left( \frac{2 \pi k l}{n} \right) \right| \\
&\geq & \sqrt{\frac{n}{2}}. 
\end{eqnarray*}

Also each term
\begin{eqnarray*}
\lambda_{n,k }\Vert h_{n,k} \Vert_{M^1} &=& \frac{1}{n^3}(1+\frac{k}{n^p}) \cdot \frac{1}{\sqrt{n}} \sum_{l=0}^{n-1} \left| \cos\left( \frac{2 \pi k l}{n} \right)+ \sin\left( \frac{2 \pi k l}{n} \right) \right|\\
& \leq & \frac{1}{n^3}(1+\frac{k}{n^p}) \cdot \sqrt{n}  < \infty.
\end{eqnarray*}
However, 
\begin{eqnarray*}
\sum_{n=1}^\infty \sum_{k=0}^{n-1} \lambda_{n,k }\Vert h_{n,k} \Vert_{M^1}^2 &\geq & \sum_{n=1}^\infty \frac{1}{2n^2} \sum_{k=0}^{n-1}(1+\frac{k}{n^p}) \\
&=& \sum_{n=1}^\infty \frac{1}{2n^2} \left( n+ \frac{n(n-1)}{2n^p} \right)\\
&\geq & \sum_{n = 1}^\infty \frac{1}{2n}=\infty.
\end{eqnarray*}

\end{proof}

\section*{Acknowledgments} R.~Balan and K.~A.~Okoudjou were partially supported by ARO grant W911NF1610008. R.~Balan was also partially supported by the NSF grant DMS-1413249 and the LTS grant H9823013D00560049. K.~A.~Okoudjou  was also partially supported by a grant from the Simons Foundation $\# 319197$. This material is partially based upon work supported by the National Science Foundation under Grant No.~DMS-1440140 while K.~A.~Okoudjou was in residence at the Mathematical Sciences Research Institute in Berkeley, California, during the Spring 2017 semester.
A.~Poria is grateful to the United States-India Educational Foundation for providing the Fulbright-Nehru Doctoral Research Fellowship, and to the Department of Mathematics, University of Maryland, College Park, USA for the support provided during the period of this work. He would also like to express his gratitude to the Norbert Wiener Center for Harmonic Analysis and Applications at the University of Maryland, College Park for its kind hospitality, and the Indian Institute of Technology Guwahati, India for its support.

\bibliographystyle{plain}

\end{document}